\documentclass{amsart}

\usepackage{graphicx}
\usepackage{amsmath}
\usepackage{amsfonts}
\usepackage{amssymb}
\usepackage{bbm}
\usepackage{epic}
\usepackage{amscd}

\usepackage{amsthm}

\usepackage[all]{xy}

\usepackage[margin=1.0in]{geometry}

\usepackage{appendix}

\usepackage{multirow}
\usepackage{tikz-cd}
\usepackage{url}

\newtheorem{theorem}{Theorem}[section]

\theoremstyle{definition}

\newtheorem{remark}{Remark}[section]

\newtheorem{conjecture}{Conjecture}[section]
\newtheorem{question}{Question}[section]
\newtheorem{corollary}{Corollary}[section]

\DeclareMathOperator{\rank}{rank}

\DeclareMathOperator{\diag}{diag}

\DeclareMathOperator{\rk}{rank}
\DeclareMathOperator{\qrk}{qrank}

\DeclareMathOperator{\qdim}{qdim}
\DeclareMathOperator{\qrnk}{qrank}

\newcommand{\Nats}{\mathbbm{N}}

\newcommand{\Reals}{\mathbbm{R}}

\newcommand{\Ints}{\mathbbm{Z}}

\DeclareMathOperator{\Vect}{Vect_{\mathbbm{F}}}
\DeclareMathOperator{\Foams}{Foams}
\DeclareMathOperator{\Tait}{Tait}
\DeclareMathOperator{\F}{\mathbbm{F}}
\DeclareMathOperator{\adm}{adm}
\DeclareMathOperator{\facets}{facets}

\begin{document}

\title{Computer Bounds for Kronheimer-Mrowka Foam Evaluation}

\author{David Boozer} 

\date{\today}

\begin{abstract}
Kronheimer and Mrowka recently suggested a possible approach
towards a new proof of the four color theorem that does not rely on
computer calculations.
Their approach is based on a functor $J^\sharp$, which they define
using gauge theory, from the category of webs and foams to the
category of vector spaces over the field of two elements.
They also consider a possible combinatorial replacement
$J^\flat$ for $J^\sharp$.
Of particular interest is the relationship between the dimension of
$J^\flat(K)$ for a web $K$ and the number of Tait colorings $\Tait(K)$
of $K$; these two numbers are known to be identical for a special
class of ``reducible'' webs, but whether this is the case for
nonreducible webs is not known.
We describe a computer program that strongly constrains the
possibilities for the dimension and graded dimension of $J^\flat(K)$
for a given web $K$, in some cases determining these quantities
uniquely.
We present results for a number of nonreducible example webs.
For the dodecahedral web $W_1$ the number of Tait colorings is
$\Tait(W_1) = 60$, but our results suggest that
$\dim J^\flat(W_1) = 58$.
\end{abstract}

\maketitle


\section{Introduction}

The four-color theorem states that the vertices of any planar graph
can be colored with no more than four colors in such a way that no
pair of adjacent vertices share the same color.
The theorem was first proven in 1976 by Appel and Haken via computer
calculations \cite{Appel}, and, though simplifications to their proof
have been made \cite{Robertson,Gonthier}, to this day no proof is
known that does not rely on computer assistance.

Recently Kronheimer and Mrowka suggested a new approach to the
four color theorem that may lead to the first non-computer-assisted
proof of this result \cite{Kronheimer-1}.
Their approach is based on a functor $J^\sharp$, which they define
using gauge theory, from the category of webs and foams to the
category of vector spaces over the field of two elements.
Kronheimer and Mrowka also consider a possible combinatorial
replacement $J^\flat$ for $J^\sharp$.
The functor $J^\flat$ was originally defined by Kronheimer and Mrowka
in terms of a list of combinatorial rules that they conjectured would
yield a well-defined result; this was later shown to be the case by
Khovanov and Robert \cite{Khovanov}.

In order to apply the functors $J^\sharp$ and $J^\flat$ to the
four-color problem, it is important to understand the relationships
between $\dim J^\sharp(K)$,  $\dim J^\flat(K)$, and the Tait number
$\Tait(K)$ for arbitrary webs $K$.
In addition, the vector spaces $J^\flat(K)$ carry a
$\Ints$-grading, and it is of interest to compute the quantum
dimensions $\qdim J^\flat(K)$ of these spaces.
We have written a computer program to calculate lower
bounds on $\dim J^\flat(K)$ and $\qdim J^\flat(K)$, which in some
cases are sufficiently strong to determine these quantities uniquely.
Our results are summarized in Table \ref{table:results}.
In particular, we get the following result:

\begin{theorem}
\label{theorem:results-exact}
For the webs $W_2$ and $W_3$ shown in Figure \ref{fig:example-webs} we
have that $\dim J^\flat(K) = \Tait(K)$.
\end{theorem}

For the dodecahedral web $W_1$ the Tait number is $\Tait(W_1) = 60$,
but our results show that $\dim J^\flat(W_1)$ must be either 58 and
60, and suggest that it is in fact 58.
We should emphasize that even if $\dim J^\flat(W_1) = 58$, this would
not invalidate Kronheimer and Mrowka's strategy for proving the
four-color theorem using gauge theory; see Remark
\ref{remark:implications}.

The paper is organized as follows.
In Section \ref{sec:background}
we describe the functors $J^\sharp$ and $J^\flat$ and their
relationship to the four-color problem.
In Section \ref{sec:program} we describe the computer program.
In Section \ref{sec:results} we present the resulting lower bounds
on $\dim J^\flat(K)$ and $\qdim J^\flat(K)$ for a number of
nonreducible example webs $K$.
In Section \ref{sec:questions} we discuss some open questions.

\section{Background}
\label{sec:background}

Kronheimer and Mrowka's new approach to the four-color problem relies
on concepts involving webs and foams, which we briefly review here.
A \emph{web} is an unoriented planar trivalent graph.
A \emph{foam} is a kind of singular cobordism between two webs.
More precisely, a \emph{closed foam} $F$ is a singular 2D surface
embedded in $\Reals^3$ in which every point $p \in F$ has a
neighborhood that takes the form of one of
three local models shown in Figure \ref{fig:local-models}.
Points with the local model shown in
Figure \ref{fig:local-models}a,
Figure \ref{fig:local-models}b, and
Figure \ref{fig:local-models}c
are called \emph{regular points}, \emph{seam points}, and
\emph{tetrahedral points}, respectively.
The set of regular points forms a smooth 2D manifold whose connected
components are the \emph{facets} of $F$.
Each facet may be decorated with a finite number (possibly zero) of
marked points called \emph{dots}.
In general, we want to consider \emph{foams with boundary}
$F \subset \Reals^2 \times [a,b]$, which have local models
$K_- \times [a,a+\epsilon)$ and $K_+ \times (b-\epsilon,b]$ for webs
$K_-$ and $K_+$ near the bottom and top of the foam.
We define a \emph{half-foam} to be a foam with bottom boundary
$K_- = \varnothing$.
We can define a category $\Foams$ with webs as objects and
isotopy classes of foams with fixed boundary as morphisms.
We will thus sometimes refer to a foam $F$ with bottom boundary $K_-$
and top boundary $K_+$ as a \emph{cobordism} $F:K_- \rightarrow K_+$.

\begin{figure}
  \centering
  \includegraphics[scale=0.7]{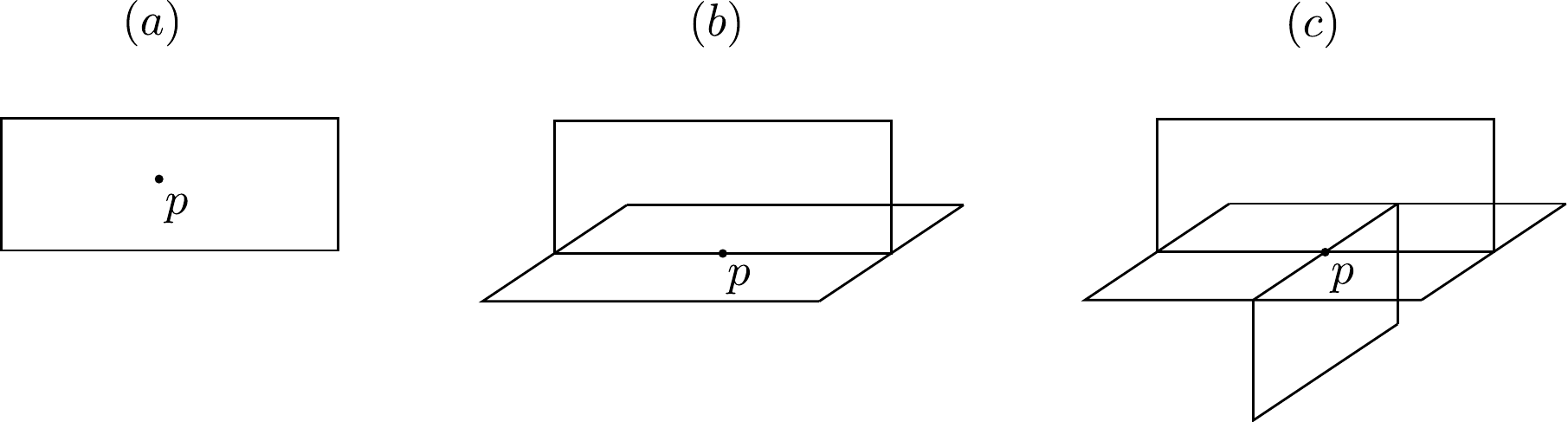}
  \caption{
    \label{fig:local-models}
    Local models for closed foams $F$.
    (a) Local model near a regular point $p \in F$.
    (b) Local model near a seam point $p \in F$.
    (c) Local model near a tetrahedral point $p \in F$.
  }
\end{figure}

Using a simple argument, the four-color theorem can be recast into the
language of webs.
We first define some additional terminology.
An edge $e$ of a web is said to be a \emph{bridge} if there is a
simple closed curve that intersects $e$ transversely in a single point
and is otherwise disjoint from the web.
A \emph{Tait coloring} of a web is a 3-coloring of the edges of the
web such that no two edges incident on a given vertex share the same
color.
Given a web $K$, we define the \emph{Tait number} $\Tait(K)$ to
be the number of distinct Tait colorings of $K$.
The four-color theorem is then equivalent to:

\begin{theorem}
\label{theorem:4-color-reformulated}
(Four-color theorem, reformulated)
For any bridgeless web $K$, we have $\Tait(K) > 0$.
\end{theorem}

This reformulation allows Kronheimer and Mrowka to introduce ideas
from gauge theory; in essence, they define a version of singular
instanton homology in which the gauge fields are required to have
prescribed singularities along a given web $K$.
In this manner they define a functor
$J^\sharp:\Foams \rightarrow \Vect$ from the category of foams to the
category of vector spaces over $\F$, the field of two
elements.
In particular, the functor associates a natural number
$\dim J^\sharp(K)$ to each web $K$.

\begin{remark}
In fact, the functor $J^\sharp$ can be defined for a more general
source category in which the webs are embedded in $\Reals^3$ and the
foams are embedded in $\Reals^4$.
We will not consider these more general notions of webs and foams
here.
\end{remark}

Kronheimer and Mrowka prove the following theorems:

\begin{theorem}
\label{theorem:nonvanishing}
(Kronheimer and Mrowka \cite[Theorem 1.1]{Kronheimer-1})
Given a web  $K$, we have $\dim J^\sharp(K) = 0$ if and only
if $K$ has a bridge.
\end{theorem}

\begin{theorem}
\label{theorem:tait-sharp}
(Kronheimer and Mrowka \cite{Kronheimer-3})
Given a web  $K$, we have $\dim J^\sharp(K) \geq \Tait(K)$.
\end{theorem}

Based on example calculations, as well as some general properties of
the functor $J^\sharp$, Kronheimer and Mrowka conjecture that

\begin{conjecture}
\label{conjecture:tait-sharp}
For any web $K$, we have $\dim J^\sharp(K) = \Tait(K)$.
\end{conjecture}

Kronheimer and Mrowka show that Conjecture
\ref{conjecture:tait-sharp} is true for a special class of
\emph{reducible} webs $K$ (these are called \emph{simple} in
\cite{Kronheimer-1}), which we define in Section
\ref{ssec:construct-gens}.

Together with Theorem \ref{theorem:nonvanishing}, Conjecture
\ref{conjecture:tait-sharp} implies Theorem
\ref{theorem:4-color-reformulated}, the reformulated four-color
theorem.
It is thus of great interest to determine whether Conjecture
\ref{conjecture:tait-sharp} is in fact true.
As a possible route towards that goal, Kronheimer and Mrowka suggested
that the gauge-theoretic functor $J^\sharp:\Foams \rightarrow \Vect$
might be related to a simpler functor
$J^\flat:\Foams \rightarrow \Vect$ that could be defined in a purely
combinatorial fashion.
Kronheimer and Mrowka descried a list of combinatorial evaluation
rules that they conjectured would assign a well-defined field element
$J^\flat(F) \in \F$ to every closed foam $F$.
This conjecture was later shown to be correct by Khovanov and Robert
\cite{Khovanov}, who  adapted ideas by Robert and Wagner \cite{Robert}
to describe an explicit formula for $J^\flat(F)$.
Kronheimer and Mrowka conjecture that

\begin{conjecture}
\label{conjecture:flat-sharp}
(Kronheimer and Mrowka \cite[Conjecture 8.10]{Kronheimer-1})
For any closed foam $F$, we have $J^\flat(F) = J^\sharp(F)$.
\end{conjecture}

Once $J^\flat(F) \in \F$ has been defined for closed foams $F$, one
can use the \emph{universal construction} \cite{Blanchet} to extend
$J^\flat$ to a functor $J^\flat:\Foams \rightarrow \Vect$.
This is done in \cite{Kronheimer-1} as follows.
For the empty web $\varnothing$, define
$J^\flat(\varnothing) = \F$.
For a nonempty web $K$, let $V(K)$ be the $\F$-vector space spanned by
all half-foams with top boundary $K$.
Define a bilinear form $(-, -):V(K) \otimes V(K) \rightarrow \F$ such
that $(F_1,F_2) = J^\flat(F_1 \cup_K \bar{F}_2)$, where
$F_1 \cup_K \bar{F}_2$ is the closed foam obtained by reflecting
$F_2$ top-to-bottom to get $\bar{F}_2$ and then gluing it to $F_1$
along $K$.
Now define $J^\flat(K)$ to be the quotient of $V(K)$ by the
orthogonal complement of $V(K)$ relative to $(-,-)$.
Given a cobordism $F_{21}:K_1 \rightarrow K_2$ from web $K_1$ to web
$K_2$, define $J^\flat(F_{21}):J^\flat(K_1) \rightarrow J^\flat(K_2)$
such that
$J^\flat(F_{21})([F_1]) = [F_{21} \cup_K F_1] \in J^\flat(K_2)$ for
$[F_1] \in J^\flat(K_1)$;
note that we have defined $J^\flat$ on webs in precisely such a way
that $J^\flat$ is well-defined on cobordisms.
A closed foam $F$ can be viewed as a cobordism
$F:\varnothing \rightarrow \varnothing$, so $J^\flat(F)$ is a linear
map from $\F$ to $\F$.
Identifying this linear map with an element of $\F$, we recover
the original closed foam evaluation $J^\flat(F) \in \F$.

\begin{remark}
By construction, the bilinear form $(-,-)$ on $V(K)$ induces a
nondegenerate bilinear form on $J^\flat(K)$ for any web $K$.
The functor $J^\sharp$ is monoidal, but it is not known if $J^\flat$
is monoidal.
\end{remark}

If Conjecture \ref{conjecture:flat-sharp} is true, then
the vector space $J^\flat(K)$ is a subquotient of
the vector space $J^\sharp(K)$, as can be seen as follows.
Consider the subspace $\tilde{V}(K)$ of $J^\sharp(K)$ spanned by all vectors
of the form $J^\sharp(F)$, where $F \in V(K)$ is a half-foam with top
boundary $K$.
The bilinear form $(-,-):V(K) \otimes V(K) \rightarrow \F$ restricts
to a bilinear form
$(-,-)_{\tilde{V}(K)}:\tilde{V}(K) \otimes \tilde{V}(K) \rightarrow \F$.
Then $J^\sharp(K)$ is the quotient of $\tilde{V}(K)$ by the orthogonal
complement of $\tilde{V}(K)$ relative to $(-,-)_{\tilde{V}(K)}$.

An easy corollary to a result of Khovanov and Robert
\cite[Proposition 4.18]{Khovanov} is the following:

\begin{corollary}
\label{cor:tait-flat}
For any web $K$, we have that $\dim J^\flat(K) \leq \Tait(K)$.
\end{corollary}

Khovanov and Robert ask the following question, which they answer in
the affirmative in the case that $K$ is reducible:

\begin{question}
\label{question:tait-flat}
(Khovanov and Robert \cite{Khovanov})
Is $\dim J^\flat(K) = \Tait(K)$ for every web $K$?
\end{question}

In summary, we can assign three natural numbers to any web $K$:
$\Tait(K)$, $\dim J^\sharp(K)$, and $\dim J^\flat(K)$.
Theorem \ref{theorem:tait-sharp} and
Corollary \ref{cor:tait-flat} imply that for any web $K$ these three
numbers are related by
\begin{align}
  \dim J^\flat(K) \leq \Tait(K) \leq \dim J^\sharp(K),
\end{align}
and for reducible webs $K$ these three numbers coincide:
\begin{align}
  \dim J^\flat(K) = \Tait(K) = \dim J^\sharp(K).
\end{align}
Conjecture \ref{conjecture:tait-sharp} states that
$\dim J^\sharp(K) = \Tait(K)$ for all webs $K$, and
Question \ref{question:tait-flat} asks whether
$\dim J^\flat(K) = \Tait(K)$ for all webs $K$.
Due to Theorem \ref{theorem:nonvanishing}, a proof that
$\dim J^\sharp(K) = \Tait(K)$ for all webs $K$ would yield a proof
of the four-color theorem.

In light of these interrelated conjectures, it is of interest to
compute examples of
$\dim J^\flat(K)$ and $\dim J^\sharp(K)$ for nonreducible webs $K$.
The only such results that have yet been obtained are for the
dodecahedral web $W_1$, the smallest nonreducible web, which has Tait
number $\Tait(W_1) = 60$.
In particular, Kronheimer and Mrowka show:

\begin{theorem}
(Kronheimer and Mrowka \cite{Kronheimer-1})
For the dodecahedral web $W_1$, we have $\dim J^\flat(W_1) \geq 58$.
\end{theorem}

\begin{theorem}
(Kronheimer and Mrowka \cite{Kronheimer-2})
For the dodecahedral web $W_1$, we have $\dim J^\sharp(W_1) \leq 68$.
\end{theorem}

Due to Theorem \ref{theorem:tait-sharp} and Corollary
\ref{cor:tait-flat}, it follows from these results that for the
dodecahedral web $W_1$ the current dimension bounds are
$58 \leq \dim J^\flat(W_1) \leq 60$ and
$60 \leq \dim J^\sharp(W_1) \leq 68$.

\begin{remark}
\label{remark:implications}
We should emphasize that even if it is the case that $\dim
J^\flat(W_1) = 58$ for the dodecahedral web $W_1$, as our results seem
to suggest, this would not invalidate Kronheimer and Mrowka's strategy
for proving the four-color theorem, since it leaves open the
possibility that $\dim J^\sharp(K) = \Tait(K)$ for all webs $K$.
Rather, if $\dim J^\flat(W_1) = 58$ this would rule out only one
possible implementation of their strategy, namely that of showing that
$J^\flat(K) = J^\sharp(K)$ for all webs $K$.
\end{remark}

Since $J^\flat(K)$ is defined in terms of an infinite number of
generators mod an infinite number of relations, it is not
clear whether $\dim J^\flat(K)$ can be algorithmically
computed.
Nevertheless, it is possible to algorithmically compute lower bounds
for $\dim J^\flat(K)$, as can be seen by considering the following
two facts.
First, Khovanov and Robert's foam evaluation formula shows that
$J^\flat(F) \in \F$ is algorithmically computable for closed foams
$F$.
Second, we have the following easy observation:

\begin{remark}
\label{remark:rank-bound}
Given a finite-dimensional subspace $W$ of the vector space $V(K)$
spanned by all half-foams with top boundary $K$, the rank of the
bilinear form $(-,-)$ restricted to $W$ bounds $\dim J^\flat(K)$ from
below.
\end{remark}

Since the restriction of the bilinear form $(-,-)$ to $W$ can be
determined by evaluating the closed foams resulting from all possible
pairings of a basis of half-foams for $W$, it follows that
lower bounds for $\dim J^\flat(K)$ can be algorithmically computed.

\section{Computer program}
\label{sec:program}

We have written a computer program in Mathematica to determine lower
bounds for $\dim J^\flat(K)$ by following the procedure described at
the end of Section \ref{sec:background}.
The program consists of two distinct components.
The first component takes as input a web $K$ and
produces as output a large set $S(K)$ of half-foams with top boundary
$K$.
The second component computes the rank of the bilinear
form $(-,-)$ restricted to the vector space $W(K)$ spanned by $S(K)$.
This computation is accomplished by applying Khovanov and Robert's
evaluation formula to the closed foams obtained by taking all possible
pairings of half-foams in $S(K)$.
The program is available from the author's website \cite{Boozer}.

\begin{remark}
\label{remark:foam-lists}
The program represents foams as lists
$(K_1 = \varnothing, E_1, K_2, E_2, \cdots, E_n, K_{n+1})$ in which
webs $K_i$ are interleaved with certain ``elementary'' cobordisms
$E_i:K_i \rightarrow K_{i+1}$ connecting adjacent pairs of webs.
The terminal web $K_{n+1}$ is $K_{n+1} = \varnothing$ for a closed
foam and $K_{n+1} = K$ for a half-foam with top boundary $K$.
The foam-representation lists can be thought of as describing
successive horizontal slices through the foam, where the web $K_i$
describes the $i$-th slice and the cobordism $E_i$ describes the
portion of the foam that lies between the $i$-th and $(i+1)$-th slice.
The allowed elementary cobordisms $E_i$ are shown in Figures
\ref{fig:elem-cobordisms-red} and \ref{fig:elem-cobordisms-nonred}.
\end{remark}

\subsection{Construction of generating set of half-foams}
\label{ssec:construct-gens}

The first component of the program takes as input a web $K$ and
produces as output a set $S(K)$ of half-foams with top boundary $K$.
The program relies on the fact that for certain ``reducible''
webs $K$ there is an algorithm for producing a basis of half-foams for
$J^\flat(K)$.

We first define the notion of reducibility.
Given a web $K$, we can consider the local replacements
shown in Figure \ref{fig:local-replacements-red} in which a
\emph{small face} of $K$, defined to be a disk, bigon, triangle, or
square, is eliminated to yield a simpler web $K'$.
We say that a web is \emph{reducible} if there is a series of such
local replacements that terminates in the empty web.
That is, the empty web is reducible, and a nonempty web $K$ is
reducible if $K_1'$ is reducible,
$K_2'$ is reducible, $K_3'$ is reducible, or $K_{4a}'$ and $K_{4b}'$
are both reducible, for some choice of local replacements of the form
shown in Figure \ref{fig:local-replacements-red}.

\begin{remark}
\label{remark:tait-local-replacements}
It is straightforward to show that the Tait number satisfies the
following relations for the local
replacements shown in Figure \ref{fig:local-replacements-red}:
\begin{enumerate}
  \item
    $\Tait(K) = 3\Tait(K_1')$.
  \item
    $\Tait(K) = 2\Tait(K_2')$.
  \item
    $\Tait(K) = \Tait(K_3')$.
  \item
    $\Tait(K) = \Tait(K_{4a}') + \Tait(K_{4b}')$.
\end{enumerate}
Here we have adopted the convention that
$\Tait(\varnothing) = 1$.
\end{remark}

\begin{figure}
  \centering
  \includegraphics[scale=0.7]{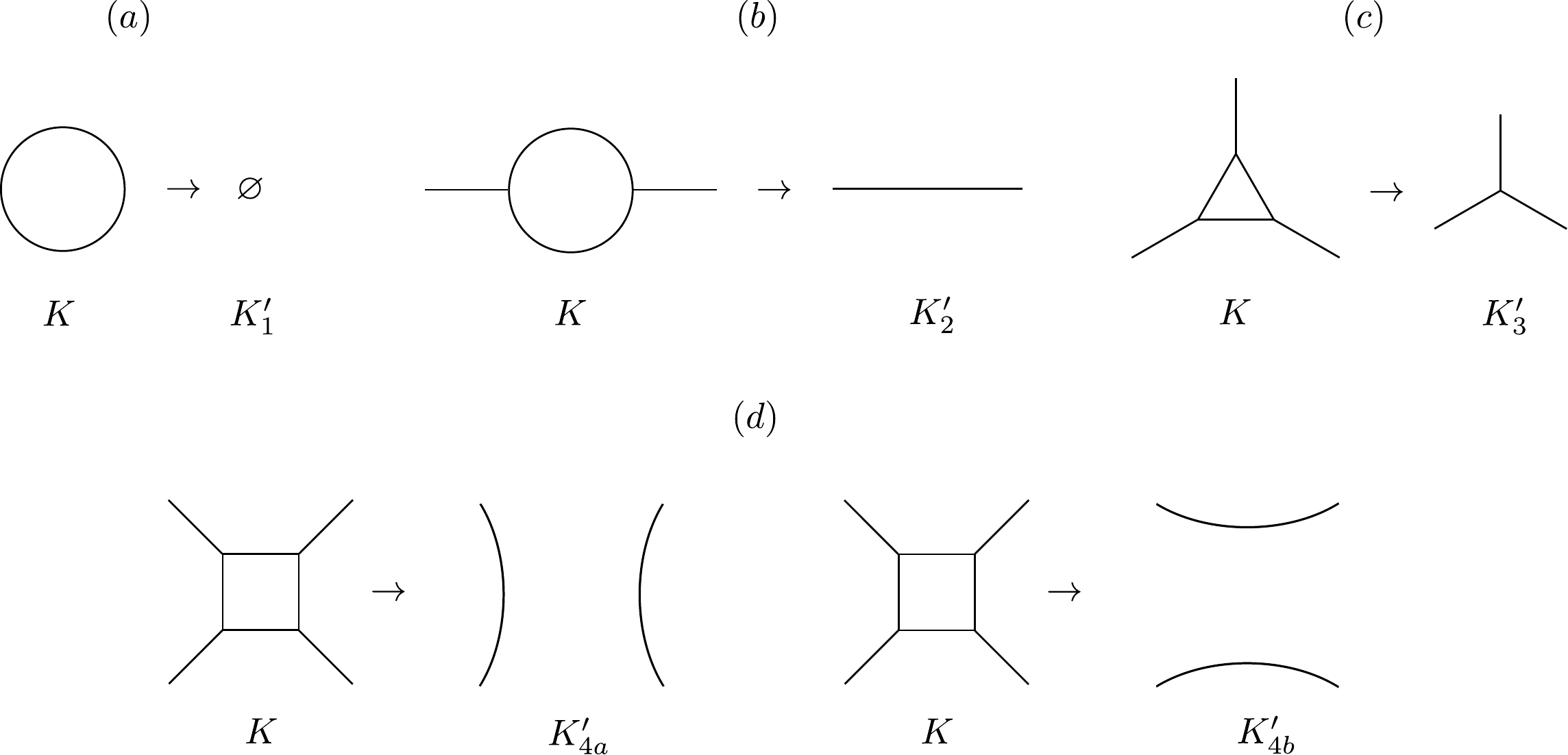}
  \caption{
    \label{fig:local-replacements-red}
    Local replacements $K \rightarrow K'$ for reducible webs.
    The web $K'$ is identical to $K$ outside the depicted region.
    (a) Disk elimination $K \rightarrow K_1'$.
    (b) Bigon elimination $K \rightarrow K_2'$.
    (c) Triangle elimination $K \rightarrow K_3'$.
    (d) Square elimination $K \rightarrow K_{4a}'$ and
    $K \rightarrow K_{4b}'$.
  }
\end{figure}

Using special properties of reducible webs, Khovanov and Robert show
that $\dim J^\flat(K) = \Tait(K)$ for reducible webs $K$
\cite{Khovanov}; moreover, their results provide an algorithm for
constructing a basis $S(K)$ of half-foams for $J^\flat(K)$.
The algorithm involves a set of elementary cobordisms shown in Figure
\ref{fig:elem-cobordisms-red}, which correspond to the local
replacements shown in Figure \ref{fig:local-replacements-red}.
For the empty web $K = \varnothing$ we take $S(K) = \varnothing$.
For a nonempty web $K$ we recursively apply the following rules:
\begin{enumerate}
  \item
  If a local replacement
  $K \rightarrow K_1'$ of the form shown in Figure
  \ref{fig:local-replacements-red}a yields a reducible web $K_1'$,
  then $S(K)$ is obtained by applying each of the elementary
  cobordisms $C_1,\dot{C}_1,\ddot{C}_1:K_1' \rightarrow K$ shown in
  Figure \ref{fig:elem-cobordisms-red}a to each half-foam in
  $S(K_1')$.

  \item
  If a local replacement $K \rightarrow K_2'$
  of the form shown in Figure \ref{fig:local-replacements-red}b
  yields a reducible web $K_2'$, then
  $S(K)$ is obtained by applying each of the elementary cobordisms
  $C_2,\dot{C}_2:K_2' \rightarrow K$ shown in Figure
  \ref{fig:elem-cobordisms-red}b to each half-foam in $S(K_2')$.

  \item
  If a local replacement $K \rightarrow K_3'$
  of the form shown in Figure \ref{fig:local-replacements-red}c yields
  a reducible web $K_3'$, then
  $S(K)$ is obtained by applying the elementary cobordism
  $C_3:K_3' \rightarrow K$ shown in Figure
  \ref{fig:elem-cobordisms-red}c to each half-foam in $S(K_3')$.

  \item
  If local replacements
  $K \rightarrow K_{4a}'$ and $K \rightarrow K_{4b}'$ of the form
  shown in Figure \ref{fig:local-replacements-red}d yield reducible
  webs $K_{4a}'$ and $K_{4b}'$, then $S(K)$ is the union of the two
  sets of half-foams obtained by applying the elementary
  cobordism $C_{4a}:K_{4a}' \rightarrow K$ to each half-foam in
  $S(K_{4a}')$ and the elementary cobordism
  $C_{4b}:K_{4b}' \rightarrow K$ to each half-foam in $S(K_{4b}')$,
  where $C_{4a}$ and $C_{4b}$ are as shown in
  Figure \ref{fig:elem-cobordisms-red}d.
\end{enumerate}
Note that the resulting basis $S(K)$ depends on the specific choices
of local replacements that we make.
The above algorithm allows for easy construction of the list
representations of half-foams described in Remark
\ref{remark:foam-lists}.

\begin{figure}
  \centering
  \includegraphics[scale=0.7]{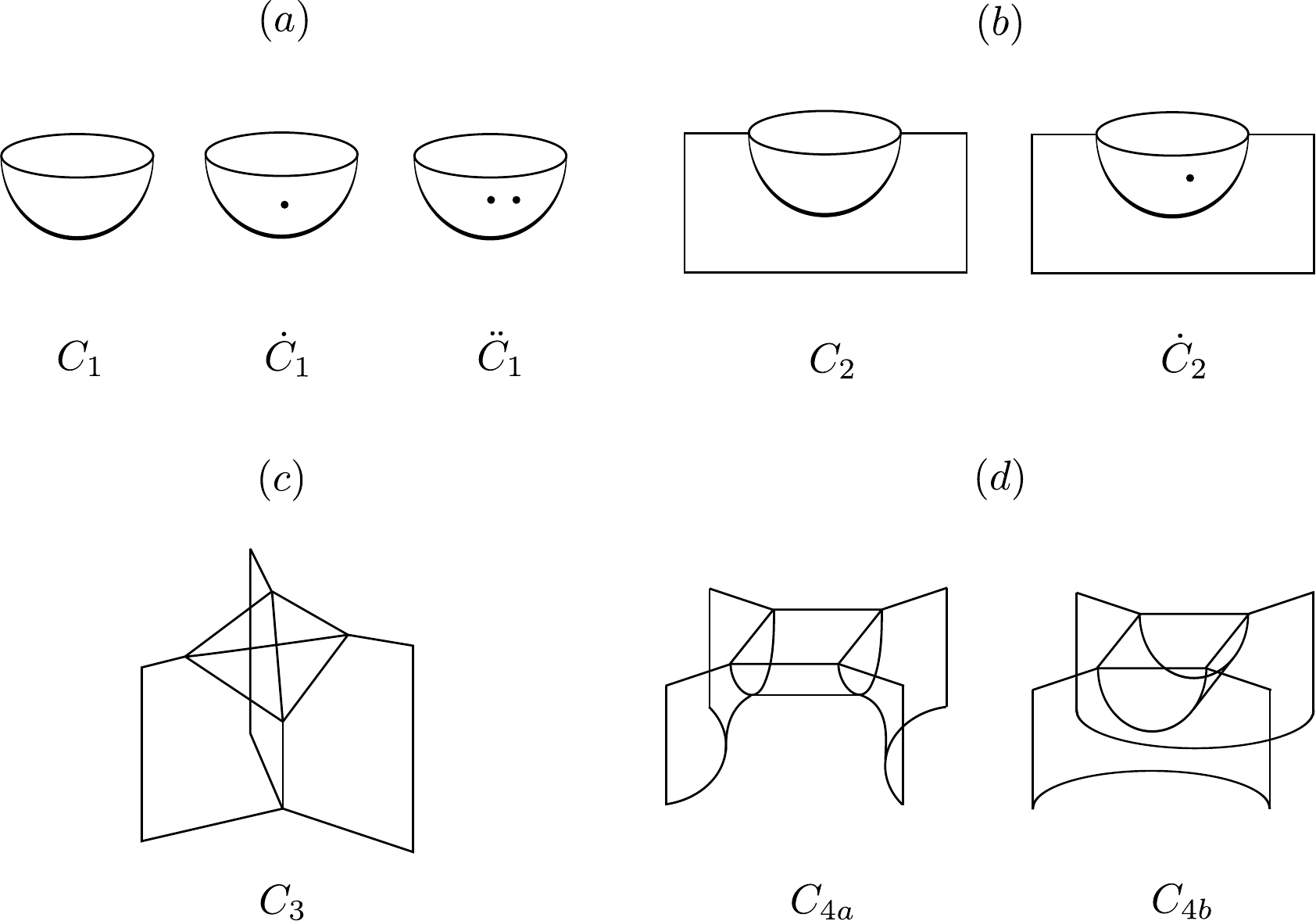}
  \caption{
    \label{fig:elem-cobordisms-red}
    Elementary cobordisms for reducible webs.
    The cobordisms are taken to be the identity outside of the
    depicted region.
    Note that the cobordisms $\dot{C}_1$, $\ddot{C}_1$, and
    $\dot{C}_2$ have facets that are decorated with dots.
  }
\end{figure}

For nonreducible webs $K$, there is no known algorithm for producing
a basis of half-foams for $J^\flat(K)$, and our goal instead is just
to produce a large set $S(K)$ of half-foams with top boundary $K$.
Ideally, we would like $S(K)$ to be large and diverse enough to
contain a spanning set for $J^\flat(K)$.
To construct the set $S(K)$, we use the fact that a nonreducible web
$K$ can often be converted into a reducible web $K'$ by making one of
the four local replacements shown in Figure
\ref{fig:local-replacements-nonred}.
If such is the case, then we can obtain a set of half-foams with
top boundary $K$ by applying the corresponding elementary cobordism
shown in Figure \ref{fig:elem-cobordisms-nonred}
to each of the half-foams in a generating set $S(K')$ constructed
using the above algorithm for reducible webs.
We construct $S(K)$ by taking the union of such sets constructed for
all local replacements $K \rightarrow K'$ of the form shown in Figure
\ref{fig:local-replacements-nonred} that yield reducible webs $K'$.
Note that the list representation, as described in Remark
\ref{remark:foam-lists}, of each half-foam in $S(K)$ has the
form $(K_1 = \varnothing, E_1, \cdots, E_n, K_{n+1} = K)$, where
for $i \in \{1, \cdots, n-1\}$ the cobordism $E_i$ one of the
cobordisms shown in Figure \ref{fig:elem-cobordisms-red} and
$E_n$ is one of the cobordisms shown in Figure
\ref{fig:elem-cobordisms-nonred}.

\begin{remark}
For simplicity, we do not actually use the above algorithm for
reducible webs to construct the generating sets $S(K')$ for
reducible webs $K'$.
Rather, given a reducible web $K'$ we apply a recursive algorithm in
which we attempt to eliminate disks, bigons, triangles, and squares,
in that order.
This algorithm is not guaranteed to reduce every reducible web
$K'$, and if it fails to do so we treat $K'$ as if it were
nonreducible.
Consequently, the generating sets $S(K)$ that we produce for
nonreducible webs $K$ may not always be as large as they could be.
\end{remark}

\begin{remark}
One could construct even larger generating sets by allowing more
complicated cobordisms than those shown in Figure
\ref{fig:elem-cobordisms-nonred}, or by applying several of these
cobordisms in succession.
However, for the example webs that we consider in Section
\ref{sec:results}, the sets $S(K)$ constructed as described above are
already sufficiently large that these generalizations seem unlikely to
yield stronger dimension bounds.
\end{remark}

\begin{figure}[t]
  \centering
  \includegraphics[scale=0.7]{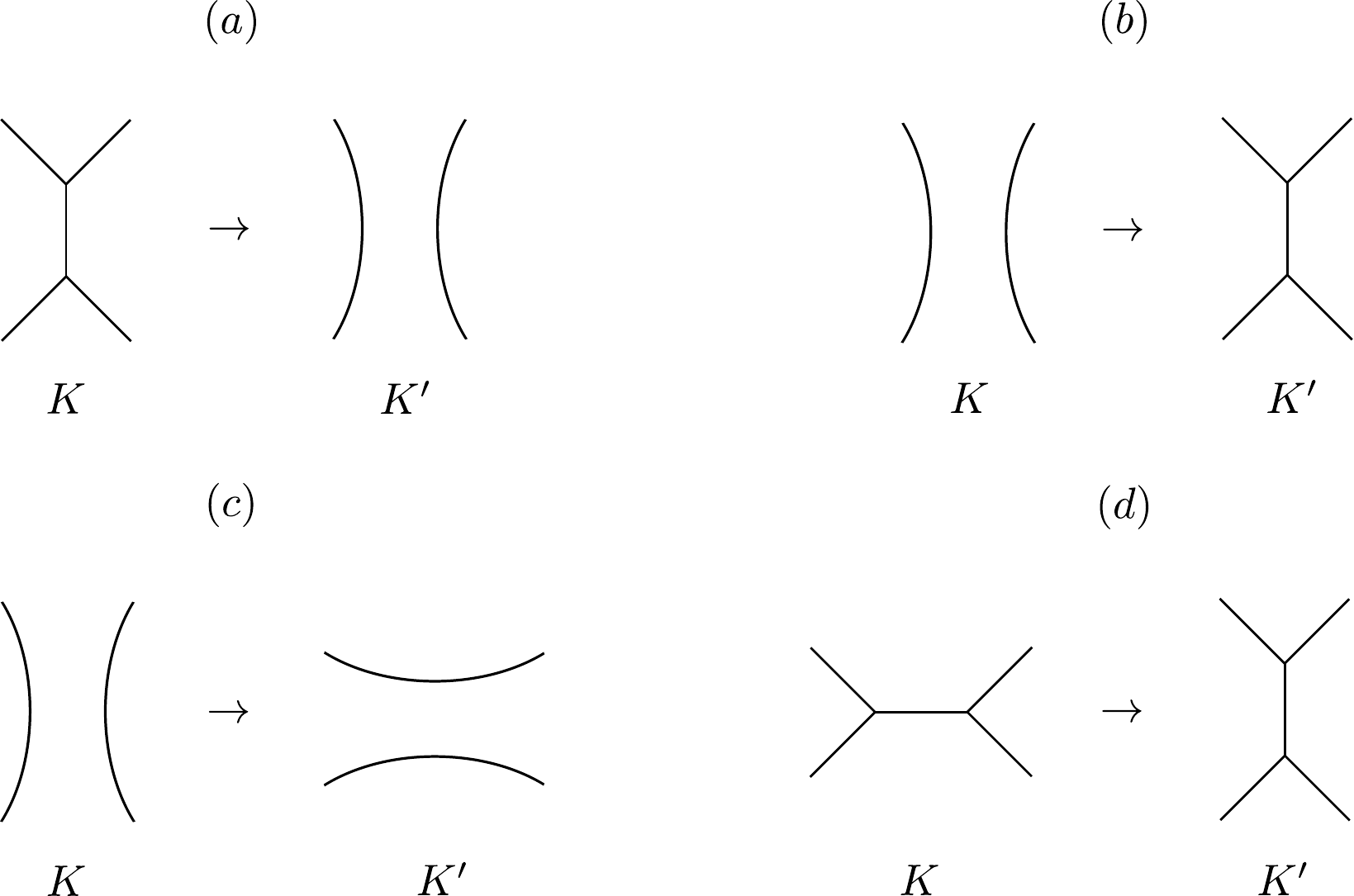}
  \caption{
    \label{fig:local-replacements-nonred}
    Local replacements $K \rightarrow K'$ for nonreducible webs.
    The web $K'$ is identical to $K$ outside the depicted region.
    (a) Zip.
    (b) Unzip.
    (c) Saddle.
    (d) IH.
  }
\end{figure}

\begin{figure}[t]
  \centering
  \includegraphics[scale=0.7]{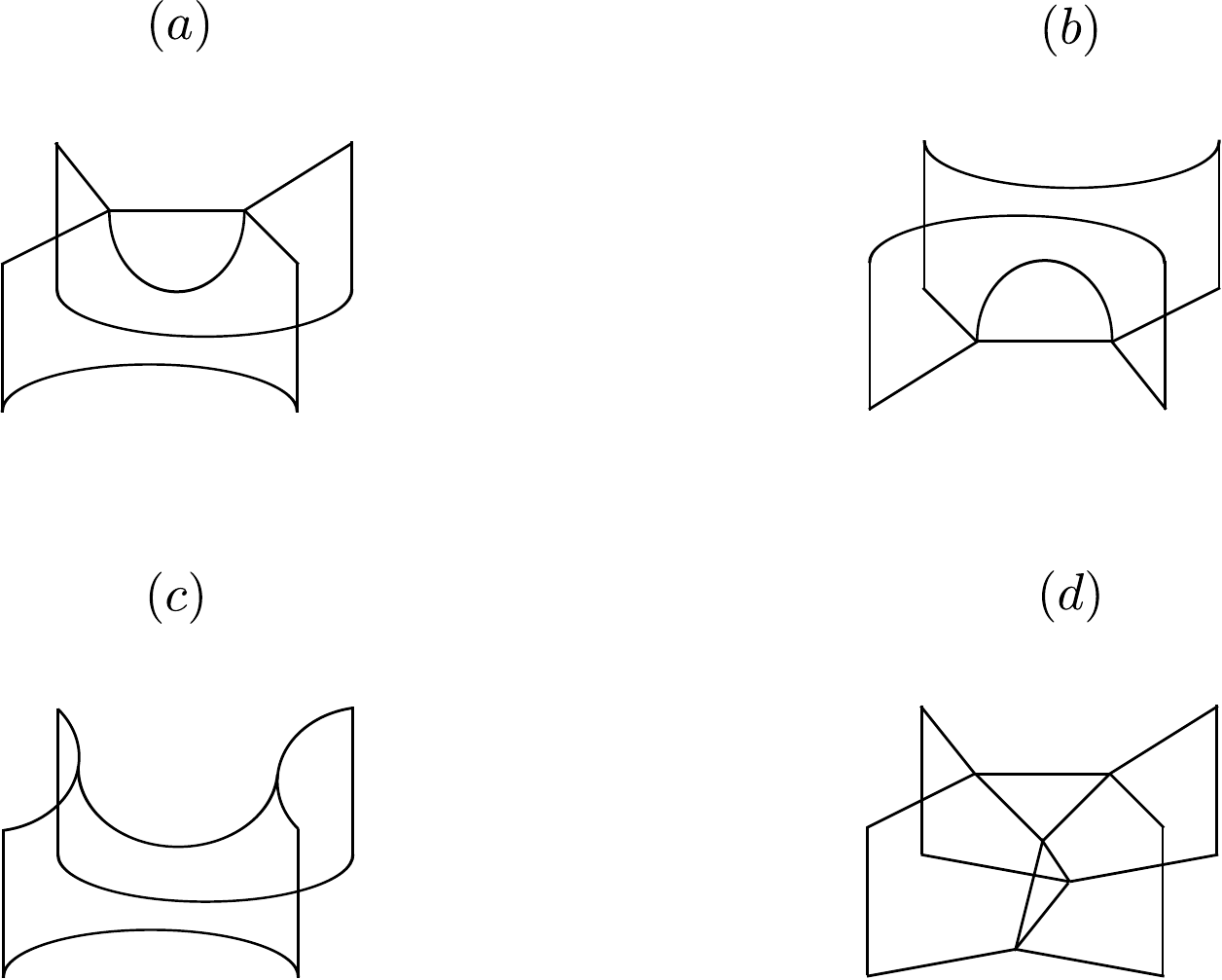}
  \caption{
    \label{fig:elem-cobordisms-nonred}
    Elementary cobordisms for nonreducible webs.
    The cobordisms are taken to be the identity outside of the
    depicted region.
    (a) Zip.
    (b) Unzip.
    (c) Saddle.
    (d) IH.
  }
\end{figure}

\subsection{Evaluation of closed foams}
\label{ssec:eval-formula}

The second component of the computer program takes as input a closed
foam $F$ and returns as output the closed foam evaluation
$J^\flat(F) \in \F$.
This component of the program is essentially a computer implementation
of Khovanov and Robert's closed foam evaluation formula
\cite{Khovanov}, which we briefly review here.

Define $\facets(F)$ to be the set of facets of a closed foam $F$.
A \emph{coloring} of a closed foam $F$ is a map
$c:\facets(F) \rightarrow \{1,2,3\}$,
and we will refer to 1, 2, and 3 as \emph{colors}.
A coloring $c$ of a foam $F$ is \emph{admissible} if the three
facets incident on any given seam of $F$ are assigned distinct colors.
Define $\adm(F)$ to be the (possibly empty) set of admissible colorings
of a closed foam $F$.

\begin{remark}
There are foams that do not have any admissible colorings (an example
is given in \cite{Khovanov}), and such foams can arise as pairs of
generating half-foams for the dodecahedral web.
Note that Khovanov and Robert define a foam to be admissible if
it is admissible in the sense that we have defined and in addition
satisfies an orientability condition; they then prove that any foam
that is admissible in the sense that we have defined automatically
satisfies the orientability condition.
\end{remark}

Define a polynomial ring in three variables $R' = \F[X_1,X_2,X_3]$ and
a ring $R'' = R'[(X_1+X_2)^{-1},(X_1+X_3)^{-1},(X_2+X_3)^{-1}]$ obtained
from $R'$ by inverting the elements $X_1+X_2$, $X_1+X_3$, and
$X_2+X_3$.
Define elementary symmetric polynomials
\begin{align}
  E_1 &= X_1 + X_2 + X_3, &
  E_2 &= X_1 X_2 + X_1 X_3 + X_2 X_3, &
  E_3 &= X_1 X_2 X_3,
\end{align}
and define $R = \F[E_1,E_2,E_3]$ to be the subring of $R'$ consisting
of symmetric polynomials in $X_1$, $X_2$, and $X_3$.
Given a closed foam $F$ with coloring $c$, define
$P(F,c) \in R'$ and $Q(F,c) \in R''$ by
\begin{align}
  \label{eqn:p}
  P(F,c) &= \prod_{f \in \facets(F)} X_{c(f)}^{d(f)}, \\
  \label{eqn:q}
  Q(F,c) &= \prod_{1 \leq i < j \leq 3}
  (X_i + X_j)^{\chi(F_{ij}(c))/2},
\end{align}
where $d(f) \in \Nats$ denotes the number of dots decorating the facet
$f$ and $\chi(F_{ij}(c))$ denotes the Euler characteristic of the
closed surface $F_{ij}(c)$ obtained by taking the closure of the union
of all the facets of $F$ colored either $i$ or $j$ by $c$.
Khovanov and Robert show that $F_{ij}(c)$ is always
orientable, so $\chi(F_{ij}(c))$ is even and thus $Q(F,c)$ is in fact
an element of $R''$.
Define a polynomial
\begin{align}
  \label{eqn:foam-eval-R}
  \langle F \rangle = \sum_{c \in \adm(F)} \frac{P(F,c)}{Q(F,c)} \in R.
\end{align}
That $\langle F \rangle$ is a polynomial, rather than just a
rational function, is not obvious from equation
(\ref{eqn:foam-eval-R}), but it is nonetheless true.
Khovanov and Robert's evaluation formula for
$J^\flat(F)$ is given by evaluating $\langle F \rangle$
at $E_1 = E_2 = E_3 = 0$:
\begin{align}
  \label{eqn:foam-eval-F}
  J^\flat(F) = \langle F \rangle |_{E_1 = E_2 = E_3 = 0} \in \F.
\end{align}
Note that from equations (\ref{eqn:foam-eval-R}) and
(\ref{eqn:foam-eval-F}) it follows that $J^\flat(F) = 0$ if $F$ has no
admissible colorings.

Our computer program implements Khovanov and Robert's closed foam
evaluation formula using the list representation of foams described in
Remark \ref{remark:foam-lists}.
Given a closed foam $F$, we
\begin{enumerate}
\item
Enumerate the facets of $F$.

\item
Compute the adjacency graph for $\facets(F)$.
The adjacency graph is an unoriented
graph with a vertex $v_i$ for each facet $f_i \in \facets(F)$ and an
edge connecting vertices $v_i$ and $v_j$ if and only if the closures
of the corresponding facets $f_i$ and $f_j$ intersect.

\item
Using the adjacency graph, enumerate the admissible colorings
of $F$.

\item
Compute $J^\flat(F) \in \F$ using equations
(\ref{eqn:p}), (\ref{eqn:q}), (\ref{eqn:foam-eval-R}), and
(\ref{eqn:foam-eval-F}).

\end{enumerate}

\begin{remark}
The most complicated step of the foam evaluation algorithm is step
(1), enumerating the facets of $F$.
To achieve this, we use the fact that the closed foams we consider
always have the form $F = F_1 \cup_K \bar{F}_2$ for half-foams $F_1$
and $F_2$ with top boundary $K$.
To obtain the facets of $F$, we enumerate the facets of $F_1$ and
$F_2$ separately and then combine the two sets of facets using a
gluing algorithm.
Given a half-foam $G$ with list representation
$(K_1 = \varnothing, E_1, \cdots, E_n, K_{n+1} = K)$, we
enumerate its facets by using a recursive algorithm to enumerate
the facets of the half-foams $G_i$ with list representations
$ (K_1 = \varnothing, E_1, \cdots, E_{i-1}, K_i)$ for
$i \in \{1, \cdots, n+1\}$.
It is interesting to note that as we work from $i=1$ to $i=n$ the
number of facets is strictly increasing, and at each step the
intersection of the facets with the web $K_i$ is a partition of
$K_i$.
The number of facets can remain constant or decrease
going from $n$ to $n+1$ if the last cobordism $E_n$ is a Saddle or
an Unzip (for example, two facets that were distinct in the foam $G_n$
can describe two pieces of the same facet in the foam $G_{n+1}$).
Also, the half-foams $G_1, G_2, \cdots, G_n$ always have admissible
colorings, but for all four possibilities of the last cobordism
$E_{n+1}$ (Zip, Unzip, Saddle, or IH), it is possible that the
half-foam $G_{n+1}$ has no admissible colorings.
\end{remark}

To any foam $F$, Khovanov and Robert associate an $\Ints$-valued
degree $\deg(F)$ that is given by
\begin{align}
  \label{eqn:deg}
  \deg(F) = 2d(F) - 2\chi(F) - \chi(s(F)),
\end{align}
where $d(F)$ is the total number of dots on $F$, $\chi(F)$ is the
Euler characteristic of $F$, and $\chi(s(F))$ is the Euler
characteristic of the union $s(F)$ of the seam points and tetrahedral
points of $F$.
We use equation (\ref{eqn:deg}) to compute the degree of each of the
elementary cobordisms shown in Figures
\ref{fig:elem-cobordisms-red} and
\ref{fig:elem-cobordisms-nonred} and display the results in Table
\ref{table:deg}.
Degree is additive under composition of foams,
\begin{align}
  \deg(F_1 \circ F_2) = \deg (F_1) + \deg (F_2), 
\end{align}
so the degree of any foam built by composing elementary cobordisms can
be computed by summing the degrees of each elementary cobordism.

\begin{table}[t]
\begin{tabular}{lrrrrrrrrrrrr}
  $E$ &
  $C_1$ & $\dot{C}_1$ & $\ddot{C}_1$ & $C_2$ & $\dot{C}_2$ & $C_3$ &
  $C_{4a}$ & $C_{4b}$ &
  Zip & Unzip & Saddle & IH \\
  $\deg(E)$ &
  -2 & 0 & 2 & -1 & 1 & 0 & 0 & 0 & 1 & 1 & 2 & 1
\end{tabular}
\medskip
\medskip
\caption{
  \label{table:deg}
  Degree $\deg(E)$ of each elementary cobordism $E$.
}
\end{table}

\begin{remark}
It is useful to view $R = \F[E_1,E_2,E_3]$ as a graded ring in which
$E_1$, $E_2$, and $E_3$ are assigned gradings $2$, $4$, and $6$,
respectively.
Khovanov and Robert show that the grading of $\langle F \rangle \in R$
for a closed foam $F$ is given by $\deg F$ when $\langle F \rangle$
is nonzero \cite[Theorem 2.17]{Khovanov}.
\end{remark}

Khovanov and Robert use the notion of foam degree to impose a grading
on the vector space $J^\flat(K)$ associated to a web $K$.
Recall that $J^\flat(K)$ is spanned by vectors of the form
$J^\flat(F)$, where $F$ is a half-foam with top boundary $K$.
We define the grading of the vector $J^\flat(F) \in J^\flat(K)$ to be
$\deg F$.
In general, given a graded finite-dimensional vector space $V$ we
define $V_i$ to be the subspace of $V$ spanned by vectors of degree
$i$ and we define the quantum dimension
$\qdim V \in \Ints[q,q^{-1}]$ of $V$ to be
\begin{align}
  \qdim V = \sum_i q^i \dim V_i.
\end{align}

The functor $J^\flat$ is a special case of a more general set of
functors that Khovanov and Robert define by evaluating closed foams in
various rings.
For our purposes it is useful to consider a functor
$\langle - \rangle_\phi$ from the category $\Foams$ to the category of
modules over the polynomial ring $\F[E]$.
We view $\F[E]$ as a graded ring with $\deg E = 6$.
For a closed foam $F$, we define
$\langle F \rangle_\phi = \phi(\langle F \rangle)$, where
$\langle F \rangle \in R = \F[E_1,E_2,E_3]$ is given by
equation (\ref{eqn:foam-eval-R}) and
$\phi:\F[E_1,E_2,E_3] \rightarrow \F[E]$ is the ring homomorphism
given by $E_1, E_2 \mapsto 0$,  $E_3 \mapsto E$.
Given a web $K$, we define $M(K)$ to be the free $\F[E]$-module
spanned by all half-foams with top boundary $K$.
We define a bilinear form
$(-,-)_\phi:M(K) \otimes M(K) \rightarrow \F[E]$ such that
$(F_1,F_2)_\phi = \langle F_1 \cup_K \bar{F}_2\rangle_\phi$.
Applying the universal construction described in
Section \ref{sec:background}, we define the $\F[E]$-module
$\langle K \rangle_\phi$ to be the quotient of $M(K)$ by the
orthogonal complement of $M(K)$ relative to $(-,-)_\phi$.
Khovanov and Robert prove the following theorem:

\begin{theorem}
\label{theorem:free-tait}
(Khovanov and Robert \cite[Proposition 4.18]{Khovanov})
For any web $K$, the $\F[E]$-module $\langle K \rangle_\phi$ is free
of rank $\Tait(K)$.
\end{theorem}

In general, given a free graded $\F[E]$-module $M$ of finite rank $r$
we chose a homogeneous basis $B = \{m_1, \cdots, m_r\}$ of $M$
and define the quantum rank $\qrk M \in \Ints[q,q^{-1}]$ of $M$ to be
\begin{align}
  \qrk M = \sum_i q^i \cdot \#\{m_k \in B \mid \deg(m_k) = i\}.
\end{align}
We have the following easy corollary to Theorem
\ref{theorem:free-tait}, which subsumes Corollary
\ref{cor:tait-flat}:

\begin{corollary}
\label{cor:tait-flat-restate}
For any web $K$, we have that $\dim J^\flat(K) \leq \Tait(K)$ and
$\qdim J^\flat(K) \leq {\qrnk \langle K \rangle_\phi}$.
\end{corollary}

\begin{proof}
Recall that $J^\flat(K)$ is the $\F$-vector space obtained by taking
the quotient of $V(K)$ by the orthogonal
complement of $V(K)$ relative to $(-,-)$, and
$\langle K \rangle_\phi$ is the $\F[E]$-module obtained by taking
the quotient of $M(K)$ by the orthogonal
complement of $M(K)$ relative to $(-,-)_\phi$.
But $(F_1,F_2) \in \F$ can be obtained by evaluating
$(F_1,F_2)_\phi \in \F[E]$ at\ $E=0$, so $(F_1,F_2) = 0$ whenever
$(F_1,F_2)_\phi = 0$.

\end{proof}

\section{Results}
\label{sec:results}

We use the computer program described in Section
\ref{sec:program} to obtain lower bounds $\ell(K)$ on the dimension of
$J^\flat(K)$ for the example webs $K$ shown in
Figure \ref{fig:example-webs}.
The results are summarized in Table \ref{table:results}.
For each web $K$, we use the algorithm described in Section
\ref{ssec:construct-gens} to construct a set
$S(K) = \{F_1, \cdots, F_N\}$ of
$N$ half-foams with top boundary $K$.
For increasing values of $n$, we compute lower bounds
$\ell_n(K)$ on $\dim J^\flat(K)$ by calculating the rank of
the bilinear form $(-, -)$ restricted to the vector space spanned by
$\{F_1, \cdots, F_n\}$, using the closed-foam evaluation algorithm
described in Section \ref{ssec:eval-formula}.
In order to obtain results in a reasonable amount of time, we compute
$\ell_n(K)$ only up to an index $n = N_e$ that is less
than the total number of half-foams $N$ that we have generated.
We note that $\ell_n(K)$ is a nondecreasing function of $n$ that
saturates at a value $\ell(K)$ for some index $n = N_\ell$; that
is, $\ell_n(K) = \ell(K)$ for $N_\ell \leq n \leq N_e$, and $N_\ell$
is the smallest index with this property.
The saturation value $\ell(K)$ is the lower bound on
$\dim J^\flat(K)$ that is listed in Table \ref{table:results}.

\begin{figure}
  \centering
  \includegraphics[scale=0.7]{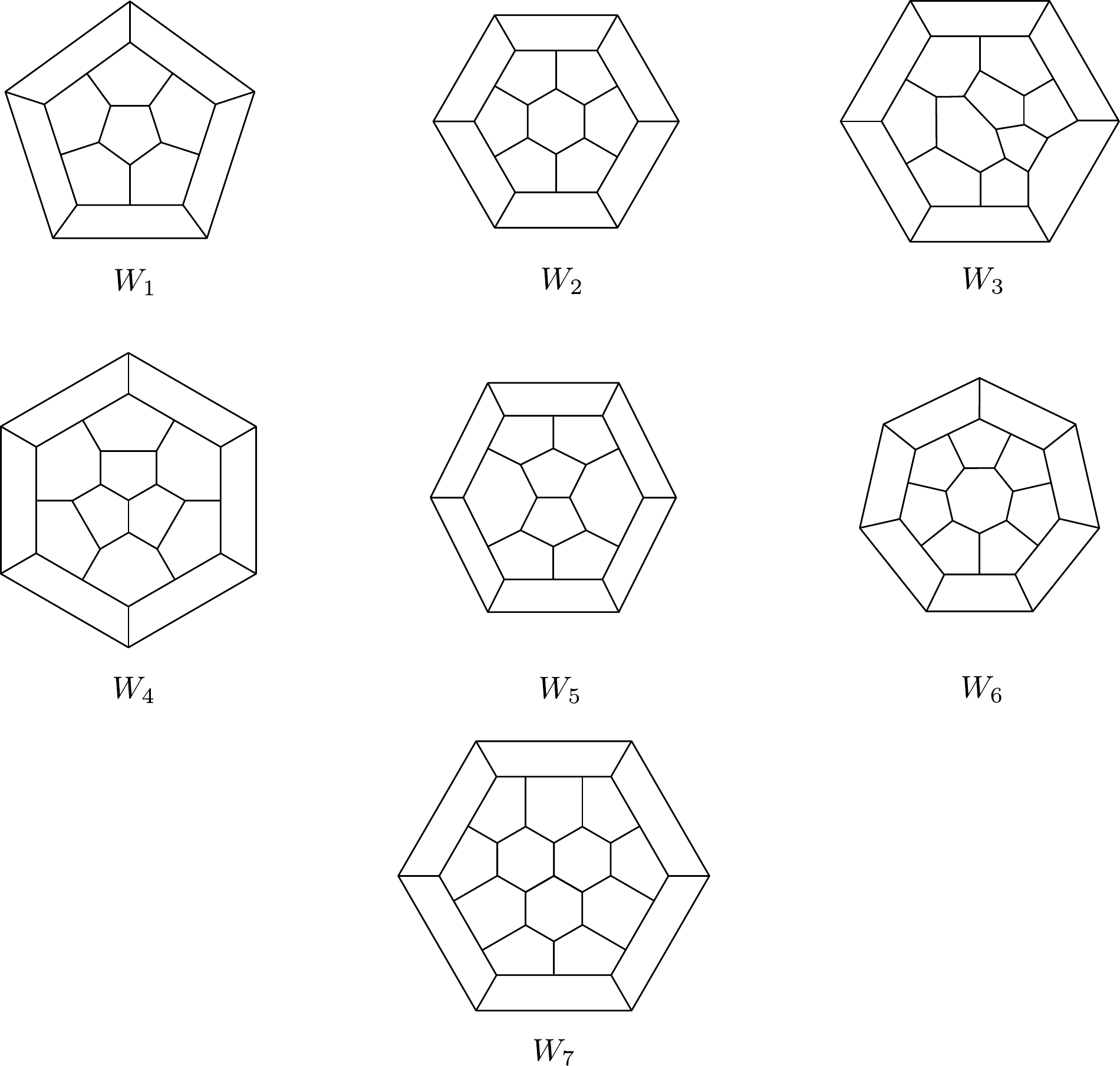}
  \caption{
    \label{fig:example-webs}
    Example webs.
    The web $W_1$ is the dodecahedral web.
  }
\end{figure}

\begin{table}
  \centering
 \begin{tabular}{|lrrrrr|}
    \hline
    Web $K$ & $\ell(K)$ & $\Tait(K)$ & $N_\ell$ & $N_e$ & $N$ \\
    \hline
    $W_1$
    & 58 & 60 &  156 & 6727 & 11\,160 \\
    $W_2$
    & 120 & 120 & 747 & 5322 & 27\,792 \\
    $W_3$
    & 162 & 162 & 822 & 4902 & 45\,960 \\
    $W_4$
    & 178 & 180 & 1193 & 6351 & 47\,196 \\
    $W_5$
    & 188 & 192 & 2447 & 7153 & 40\,704 \\
    $W_6$
    & 248 & 252 & 1726 & 6331 & 53\,172 \\
    $W_7$
    & 308 & 312 & 1190 & 5458 & 101\,970 \\
    \hline
 \end{tabular}
 \medskip
 \medskip
  \caption{
    Lower bounds $\ell(K)$ on $\dim J^\flat(K)$ and Tait number
    $\Tait(K)$ for example webs $K$.
    The numbers $N_\ell$, $N_e$, and $N$ are explained in the main
    text.
    \label{table:results}
  }
 \end{table}

\begin{remark}
A useful class of example webs to consider is the class of
\emph{fullerene graphs}; these are planar trivalent graphs with 12
pentagonal faces and an arbitrary number of hexagonal faces.
Because they contain no small faces, fullerene graphs are
always nonreducible.
A computer program for enumerating fullerene graphs is described in
\cite{Brinkmann}.
The webs $W_1$, $W_2$, and $W_5$ are the unique fullerene graphs
with 20, 24, and 26 vertices, respectively; $W_3$ and $W_4$ are the
two fullerene graphs with 28 vertices; and $W_7$ is one of 6
fullerene graphs with 34 vertices.
\end{remark}

As an example, consider the dodecahedral web $W_1$ shown in Figure
\ref{fig:example-webs}.
A graph of $\ell_n(W_1)$ versus $n$ is shown in Figure \ref{fig:graph}.
The Tait number of $W_1$ is $\Tait(W_1) = 60$, the lower bound on
$\dim J^\flat(W_1)$ computed by our program is $\ell(W_1) = 58$, and this
bound is attained after examining $N_\ell = 156$ of the $N=11\,160$
half-foams that we constructed via the algorithm described in Section
\ref{ssec:construct-gens}.
The lower bound remains $58$ even after examining
$N_e = 6\,727$ of the $N=11\,160$ half-foams we constructed, which
suggests that $\dim J^\flat(W_1) = 58$.
If so, this would answer Question
\ref{question:tait-flat} in the negative.

\begin{figure}[t]
  \centering
  \includegraphics{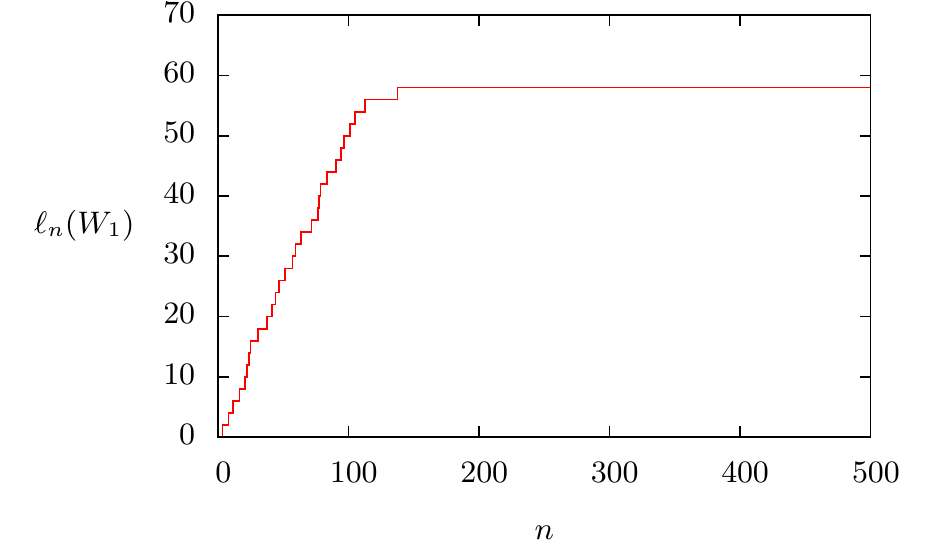}
  \caption{
    \label{fig:graph}
    Graph of $\ell_n(W_1)$ versus $n$ for the dodecahedral web $W_1$.
    }
\end{figure}

\begin{remark}
We still obtain the lower bound $\ell(W_1) = 58$ for the dodecahedral
web $W_1$ even if we use a smaller generating set $S(W_1)$ consisting
of half-foams constructed using only one of the four types of
cobordisms shown in Figure \ref{fig:elem-cobordisms-nonred}; that is,
if we use only Zip cobordisms we obtain $\ell(W_1) = 58$, if we use
only Unzip cobordisms we obtain $\ell(W_1) = 58$, if we use only
Saddle cobordisms we obtain $\ell(W_1) = 58$, and if use only IH
cobordisms we obtain $\ell(W_1) = 58$.
This fact provides additional evidence that perhaps
$\dim J^\flat(W_1) = 58$.
\end{remark}

\begin{remark}
Kronheimer and Mrowka also obtain the lower bound $\ell(W_1) = 58$ for
the dodecahedral web $W_1$ \cite{Kronheimer-1}.
Their lower bound is obtained in a manner similar to ours, but with
a generating set of half-foams constructed as follows.
Given a 4-coloring $c_4$ of the faces of a web $K \subset S^2$;
that is, a map
$c_4:\{\textup{faces of $K$}\} \rightarrow \{1,2,3,4\}$ such that no
two adjacent faces share the same color, let $T$ denote the union of
the faces of $K$ that are not colored $4$ and define an undotted
half-foam
\begin{align}
  F(K,c_4) =
  (T \times \{0\}) \cup (K \times [0,1]) \subset S^2 \times \Reals.
\end{align}
There are 240 distinct 4-colorings of the faces of the dodecahedral
web $W_1$, corresponding to 20 distinct half-foams
$\{F(W_1,c_4)\}$, each of which has degree -3.
To each of the half-foams $\{F(W_1,c_4)\}$ one can add 0, 1, 2, or 3
dots to obtain half-foams in degrees -3, -1, 1, or 3.
The resulting generating set of dotted half-foams yields the lower
bound $\ell(W_1) = 58$.
We computed lower bounds $\ell(K)$ for the example webs $W_1$, $W_2$,
$W_3$, $W_4$, $W_5$, and $W_6$
shown in Figure \ref{fig:example-webs} using generating sets
constructed in a similar manner, but, except for the dodecahedral web
$W_1$, the bounds we obtained by this method are strictly weaker than
those shown in Table \ref{table:results}.
\end{remark}

It is also of interest to compute lower bounds
on the quantum dimension $\qdim J^\flat(K)$ for each example web $K$.
As before, for each web $K$ the computer program generates a large
set $S(K)$ of half-foams with top boundary $K$.
Recall that we defined $V(K)$ and $M(K)$ to be the $\F$-vector
space and free $\F[E]$-module spanned by all half-foams with top
boundary $K$.
The set of half-foams $S(K)$ spans an $\mathbbm{F}$-vector space
$W(K) \subset V(K)$ and a free $\mathbbm{F}[E]$-module
$N(K) \subset M(K)$.
We define orthogonal complement spaces
$V(K)^\perp$, $W(K)^\perp$, $M(K)^\perp$, and $N(K)^\perp$ by
\begin{align}
  V(K)^\perp &=
  \{v \in V(K) \mid \textup{$(v,w) = 0$ for all $w \in V(K)$}\}
  \subseteq V(K), \\
  W(K)^\perp &=
  \{v \in W(K) \mid \textup{$(v,w) = 0$ for all $w \in W(K)$}\}
  \subseteq W(K), \\
  M(K)^\perp &=
  \{v \in M(K) \mid \textup{$(v,w)_\phi = 0$ for all $w \in M(K)$}\}
  \subseteq M(K), \\
  N(K)^\perp &=
  \{v \in N(K) \mid \textup{$(v,w)_\phi = 0$ for all $w \in N(K)$}\}
  \subseteq N(K).
\end{align}
Then
\begin{align}
  J^\flat(K) &= V(K)/V(K)^\perp, &
  \langle K \rangle_\phi = M(K)/M(K)^\perp.
\end{align}
We define $\ell(K)$ and $\ell_q(K)$ to be the dimension and quantum
dimension of $W(K)/W(K)^\perp$:
\begin{align}
  \ell(K) &= \dim (W(K)/W(K)^\perp), &
  \ell_q(K) &= \qdim (W(K)/W(K)^\perp).
\end{align}
The $\F[E]$-module $N(K)/N(K)^\perp$ is free, as can be shown using an
argument similar that used in
\cite[Proposition 4.4]{Khovanov} and the fact that $\F[E]$ is a PID,
and we define $r(K)$ and $r_q(K)$ to be its rank and quantum rank:
\begin{align}
  r(K) &= \rk (N(K)/N(K)^\perp), &
  r_q(K) &= \qrk (N(K)/N(K)^\perp).
\end{align}
By the same reasoning that yields Corollary
\ref{cor:tait-flat-restate}, we have that
\begin{align}
  \ell(K) \leq r(K), &&
  \ell_q(K) \leq r_q(K).
\end{align}
We have the following generalization of Remark
\ref{remark:rank-bound}:

\begin{theorem}
\label{theorem:lq-ineq}
For all webs $K$ we have that
$\ell_q(K) \leq \qdim J^\flat(K)$.
\end{theorem}

\begin{proof}
Let $\bar{W}(K) = W(K)/V(K)^\perp$ denote the image of $W(K)$ in
the quotient space $J^\flat(K) = V(K)/V(K)^\perp$.
It is clear that
\begin{align}
  \label{eqn:qrk-1}
  \qdim \bar{W}(K) \leq \qdim J^\flat(K).
\end{align}
Since $W(K) \cap V(K)^\perp \subseteq W(K)^\perp$, we have a
surjective map $\bar{W}(K) \rightarrow W(K)/W(K)^\perp$, and thus
\begin{align}
  \label{eqn:qrk-2}
  \qdim \bar{W}(K) \geq \qdim W(K)/W(K)^\perp = \ell_q(K).
\end{align}
Equations (\ref{eqn:qrk-1}) and (\ref{eqn:qrk-2}) yield the desired
result.
\end{proof}

For each example web $K$, we compute $\ell_q(K)$ and $r_q(K)$ as
follows.
We first consider $r_q(K)$.
We enumerate the spanning set $S(K)$ of $N(K)$ as
$S(K) = \{F_1, \cdots, F_n\}$ and define an $n \times n$ matrix $A$
whose matrix elements are given by
\begin{align}
  \label{eqn:smith-1a}
  A_{ij} = (F_i,F_j)_\phi \in \F[E].
\end{align}
Note that $A_{ij}$ is either zero or a nonnegative power of $E$.
We perform a Smith decomposition of $A$ to express it in the form
\begin{align}
  \label{eqn:smith-2a}
  A = SBT,
\end{align}
where $S$ and $T$ are invertible $n \times n$ matrices and $B$ is a
diagonal $n \times n$ matrix in which each matrix element along the
diagonal is either zero or a nonnegative power of $E$:
\begin{align}
  \label{eqn:smith-0}
  B = \diag\{E^{r_1}, \cdots, E^{r_m}, 0, \cdots, 0\}.
\end{align}
From equations (\ref{eqn:smith-1a}) and (\ref{eqn:smith-2a}), it
follows that
\begin{align}
  \label{eqn:smith-1}
  (F_i,F_j)_\phi = \sum_{k=1}^n S_{ik} B_{kk} T_{kj}.
\end{align}
Define sets of free homogeneous generators $\{g_1, \cdots, g_m\}$ and
$\{\tilde{g}_1, \cdots, \tilde{g}_m\}$ for $N(K)/N(K)^\perp$ by
\begin{align}
  \label{eqn:smith-2}
  g_k &= \sum_{i=1}^n (S^{-1})_{ki} F_i, &
  \tilde{g}_k &= \sum_{j=1}^n (T^{-1})_{jk} F_j.
\end{align}
From equations (\ref{eqn:smith-0}), (\ref{eqn:smith-1}), and
(\ref{eqn:smith-2}) it follows that
\begin{align}
  \label{eqn:smith-3}
  (g_i, \tilde{g}_j)_{\phi} &=
  \left\{
  \begin{array}{ll}
    E^{r_i} & \quad \mbox{if $i=j$,} \\
    0 &
    \quad \mbox{otherwise.}
  \end{array}
  \right.
\end{align}
The degree of the generator $g_k$ is given by
\begin{align}
  \deg g_k = \deg {(S^{-1})_{ki}} + \deg F_i
\end{align}
for any value of $i$ for which $(S^{-1})_{ki}$ is nonzero, and
the degree of the generator $\tilde{g}_k$ is given by
\begin{align}
  \deg \tilde{g}_k = \deg {(T^{-1})_{jk}} + \deg F_j
\end{align}
for any value of $j$ for which $(T^{-1})_{jk}$ is nonzero.
From equation (\ref{eqn:smith-3}) it follows that
\begin{align}
  \deg g_i + \deg \tilde{g}_i = \deg E^{r_i} = 6r_i.
\end{align}
The quantities $r(K)$ and $r_q(K)$ are then given by
\begin{align}
  r(K) &= \rank N(K)/N(K)^\perp = m, &
  r_q(K) &= \qrk N(K)/N(K)^\perp = \sum_{i=1}^m q^{\deg g_i}.
\end{align}
Generating sets for the vector space $W(K)/W(K)^\perp$ can be obtained
in a similar fashion by evaluating the matrices $S$, $B$, and $T$ at
$E=0$, and these generating sets can be used to compute $\ell(K)$ and
$\ell_q(K)$.

The computer results are summarized in Table \ref{table:results-qdim}
for the example webs shown in Figure \ref{fig:example-webs}.
These results give a proof of
Theorem \ref{theorem:results-exact} from the Introduction, which we
restate in a stronger form here:

\begin{theorem}
For the webs $W_2$ and $W_3$ shown in Figure \ref{fig:example-webs} we
have that
$\dim J^\flat(K) = \rk {\langle K \rangle_\phi} = \Tait(K)$ and
$\qdim J^\flat(K) = \qrk {\langle K \rangle_\phi} = \ell_q(K)$ for
$\ell_q(K)$ as shown in Table \ref{table:results-qdim}.
\end{theorem}

\begin{proof}
From Corollary \ref{cor:tait-flat-restate} and Theorem
\ref{theorem:lq-ineq}, we have that
\begin{align}
  \label{eqn:ineq-exact-result}
  \ell_q(K) \leq \qdim J^\flat(K) \leq \qrk {\langle K \rangle_\phi}.
\end{align}
Recall from Theorem \ref{theorem:free-tait} that
$\rk {\langle K \rangle_\phi} = \Tait(K)$.
Table \ref{table:results-qdim} shows that for the webs $W_2$ and $W_3$
we have that $\ell(K) = \Tait(K)$, so
$\ell(K) = \rk {\langle K \rangle_\phi}$ and thus the inequalities in
equation (\ref{eqn:ineq-exact-result}) must in fact be equalities:
\begin{align}
  \ell_q(K) = \qdim J^\flat(K) = \qrk {\langle K \rangle_\phi}.
\end{align}
\end{proof}

For the remaining webs $W_1$, $W_4$, $W_5$, $W_6$, and $W_7$,
in which $\ell(K)$ is strictly less than $\Tait(K)$, the computer
bounds strongly constrain the possibilities for $\qdim J^\flat(K)$ and
$\qdim {\langle K \rangle_\phi}$, as can be understood as follows.
Define $\bar{N}(K) = N(K)/M(K)^\perp$ to be the image of $N(K)$ in the
quotient $\langle K \rangle_\phi = M(K)/M(K)^\perp$.
In all of the example webs $K$ that we have considered, the set $S(K)$
is sufficiently large that $r(K) = \Tait(K)$ and we can thus apply
the following result:

\begin{theorem}
If $r(K) = \Tait(K)$ then $\bar{N}(K)$ is a free submodule
of $\langle K \rangle_\phi$ of full rank, with quantum rank
$\qrk \bar{N}(K) = r_q(K)$.
\end{theorem}

\begin{proof}
Recall from Theorem \ref{theorem:free-tait} that
$\langle K \rangle_\phi$ is a free $\F[E]$-module with
$\rk {\langle K \rangle_\phi} = \Tait(K)$.
Since $\bar{N}(K)$ is a submodule of
$\langle K \rangle_\phi$ and $\F[E]$ is a PID, it follows that
$\bar{N}(K)$ is free with $\rk \bar{N}(K) \leq \Tait(K)$.
Since $N(K) \cap M(K)^\perp \subseteq N(K)^\perp$, we have a
surjective homomorphism $\psi:\bar{N}(K) \rightarrow N(K)/N(K)^\perp$.
It follows that
$\rk \bar{N}(K) \geq \rk  N(K)/N(K)^\perp = r(K) = \Tait(K)$, and thus
$\rk \bar{N}(K) = \Tait(K)$.
Since $\psi$ clearly preserves degrees, it follows that
$\qrk \bar{N}(K) = \qrk N(K)/N(K)^\perp = r_q(K)$.
\end{proof}

Since $\bar{N}(K)$ is a free submodule of $\langle K \rangle_\phi$ of
full rank, the only possible difference between
$\bar{N}(K)$ and $\langle K \rangle_{\phi}$ is that homogeneous
generators of $\bar{N}(K)$ could be shifted upwards in grading by
multiples of $\deg E = 6$ relative to corresponding generators of
$\langle K \rangle_\phi$.
As an example, consider the dodecahedral web $W_1$, for which
\begin{align}
  \ell_q(W_1) &= 9 q^{-3} + 20 q^{-1} + 20 q + 9 q^3, &
  r_q(W_1) &= 9 q^{-3} + 20 q^{-1} + 20 q + 11 q^3.
\end{align}
There are only two possible cases.
One case is that $\bar{N}(W_1) = \langle W_1 \rangle_\phi$.
In this case $\dim J^\flat(W_1) = 58$ and
\begin{align}
  \qdim J^\flat(W_1) &=
  9 q^{-3} + 20 q^{-1} + 20 q + 9 q^3, &
  \qrk {\langle W_1 \rangle_\phi} &=
  9 q^{-3} + 20 q^{-1} + 20 q + 11 q^3.
\end{align}
The second case is that $\bar{N}(W_1)$ is a proper submodule of
$\langle W_1 \rangle_\phi$ of full rank.
Since
$\ell_q(W_1) \leq \qdim J^\flat(W_1) \leq
\qrk {\langle W_1 \rangle_\phi}$, one of the degree 3
generators of $\bar{N}(W_1)$ must be shifted upwards in degree
relative to a corresponding generator of $\langle W_1 \rangle_\phi$ in
degree -3.
In this case $\dim J^\flat(W_1) = 60$ and
\begin{align}
  \qdim J^\flat(W_1) &=
  \qrk {\langle W_1 \rangle_\phi} =
  10 q^{-3} + 20 q^{-1} + 20 q + 10 q^3.
\end{align}

\begin{table}
  \scalebox{0.95}{
 \begin{tabular}{|lrrlr|}
    \hline
    $K$ & $\ell(K)$ & $\Tait(K)$ &
    $\ell_q(K) + (r_q(K) - \ell_q(K))$ & \\
    \hline
    $W_1$
    & 58 & 60 &
    $9 q^{-3} + 20 q^{-1} + 20 q + 9 q^3 + (2q^3)$ & \\
    $W_2$
    & 120 & 120 &
    $3 q^{-5} + 2 q^{-4} + 16 q^{-3} + 6 q^{-2} + 29 q^{-1} + 8 + 29 q
    + 6 q^2 + 16 q^3 + 2 q^4 + 3 q^5$ & \\
    $W_3$
    & 162 & 162 &
    $2 q^{-5} + 7 q^{-4} + 13 q^{-3} + 21 q^{-2} + 24 q^{-1} + 28 +
    24 q + 21 q^2 + 13 q^3 + 7 q^4 + 2 q^5$ & \\
    $W_4$
    & 178 & 180 &
    $q^{-6} + 11 q^{-4} + 10 q^{-3} + 29 q^{-2} + 19 q^{-1} + 38 +
    19 q + 29 q^2 + 10 q^3 + 11 q^4 + q^6 + (q + q^5)$ & \\
    $W_5$
    & 188 & 192 &
    $4 q^{-5} + 31 q^{-3} + 59 q^{-1} + 59 q + 31 q^3 + 4 q^5 +
    (q + 2q^3 + q^5)$ & \\
    $W_6$
    & 248 & 252 &
    $20 q^{-4} + 62 q^{-2} + 84 + 62 q^2 + 20 q^4 + (2q^2 + 2q^4)$ & \\
    $W_7$
    & 308 & 312 &
    $4 q^{-5} + 5 q^{-4} + 41 q^{-3} + 15 q^{-2} + 79 q^{-1} + 20 +
    79 q + 15 q^2 + 41 q^3 + 5 q^4 + 4 q^5 + (q + 2q^3 + q^5)$ & \\
    \hline
 \end{tabular}
}
 \medskip
 \medskip
  \caption{
    \label{table:results-qdim}
    For each web $K$, we list $\ell(K)$, $\Tait(K) = r(K)$,
    $\ell_q(K)$, and (indicated in parentheses)
    the difference $r_q(K) - \ell_q(K)$ when this quantity is
    nonzero.
  }
 \end{table}

\section{Questions}
\label{sec:questions}

We conclude with three open questions.
We note that $\qrk {\langle K \rangle_\phi}$ is symmetric under
$q \rightarrow 1/q$ for all reducible webs $K$, and also
for webs $W_2$ and $W_3$ in Table \ref{table:results-qdim} for
which $\ell(K) = \Tait(K)$.
We can ask if this property holds for all webs:

\begin{question}
\label{question:q1}
Is it the case that $\qrk {\langle K \rangle_\phi}$ is symmetric under
$q \rightarrow 1/q$ for all webs $K$?
\end{question}

If Question \ref{question:q1} were to be answered in the affirmative,
it would imply that $\dim J^\flat(K) = \Tait(K)$ and
$\qdim J^\flat(K) = \qrk {\langle K \rangle_\phi}$ for all webs $K$,
thus answering Question \ref{question:tait-flat} in the affirmative,
due to the following result:

\begin{theorem}
If $\qrk {\langle K \rangle_\phi}$ is symmetric under
$q \rightarrow 1/q$ then
$\dim J^\flat(K) = \Tait(K)$ and
$\qdim J^\flat(K) = \qrk {\langle K \rangle_\phi}$.
\end{theorem}

\begin{proof}
Given a finite generating set of half-foams for
$\langle K \rangle_\phi$, we can proceed as in the above
discussion of the computation of $r_q(K)$ to obtain two sets of
free homogeneous generators $\{g_1, \cdots, g_r\}$ and
$\{\tilde{g}_1, \cdots, \tilde{g}_r\}$ for $\langle K \rangle_\phi$,
where $r = \Tait(K)$, such
that $(g_i,\tilde{g}_j)_\phi$ is nonzero if $i=j$ and zero otherwise.
We will say that $g_i$ and $\tilde{g}_i$ \emph{pair} together.
The claim is equivalent to the statement that only generators of
opposite degrees pair together.
Assume for contradiction that this is not the case, and
pick the largest value of $n$ such that a generator in degree $-n$
pairs with a generator in degree $n + 6m$ for $m > 0$.
By the symmetry hypothesis, the number of generators in degree
$-(n + 6m)$ is the same as the number of generators in degree
$n + 6m$, and by our choice of $n$ the generators in these opposite
degrees must mutually pair together; contradiction.
\end{proof}

For all webs $K$, except the empty web and circle web, the integer
$\Tait(K)$ is divisible by $3! = 6$, since we can permute the colors
of any Tait coloring of $K$ to obtain another Tait coloring.
Since $\rk {\langle K \rangle_\phi} = \Tait(K)$, we can ask whether
the quantum analog of this divisibility property holds:

\begin{question}
\label{question:q2}
Is $\qrk {\langle K \rangle_\phi}$ divisible by
$[3]! = (q^2 + 1 + q^{-2})(q + q^{-1})$ for all webs $K$?
\end{question}

\begin{remark}
In general, the quantum analog of a positive integer $n$ is
$[n] = q^{n-1} + q^{n-3} + \cdots + q^{-(n-1)}$ and the quantum analog
of $n!$ is $[n]! = [n][n-1]\cdots[1]$.
\end{remark}

We note that $\qrk {\langle K \rangle_\phi}$ is divisible by
$[3]!$ for all reducible webs $K$ (except the empty web
and the circle web), and for the webs $W_2$ and $W_3$
in Table \ref{table:results-qdim} for which $\ell(K) = \Tait(K)$.
For the remaining webs in Table \ref{table:results-qdim}, for which
$\ell(K) < \Tait(K)$, the computation of $r_q(K)$ shows that
if $\dim J^\flat(K) = \Tait(K)$ then
$\qdim J^\flat(K) = \qrk {\langle K \rangle_\phi}$ is divisible by
$[3]!$.
If Question \ref{question:q2} were to be answered in the affirmative,
it would have a number of implications; for example, for the
dodecahedral web $W_1$ it would force $\dim J^\flat(W_1) = 60$.

From Table \ref{table:results-qdim}, we note that the quantum
dimension $\ell_q(K)$ contains only odd powers of $q$ for webs
$W_1$ and $W_5$; only even powers of $q$ for web $W_6$; and
both even and odd powers of $q$ for webs
$W_2$, $W_3$, $W_4$, and $W_7$.
We can ask:

\begin{question}
Under what conditions does $\ell_q(K)$ contain only even, or only odd,
powers of $q$?
\end{question}

\section*{Acknowledgments}

The author would like to express his gratitude towards Ciprian
Manolescu for providing invaluable guidance, and towards Andrea
Bertozzi for the use of the UCLA Joshua computing cluster.
The author would also like to thank Mikhail Khovanov, Peter
Kronheimer, Tomasz Mrowka, and Louis-Hadrien Robert for providing many
helpful comments on an earlier version of this paper that led to
significant changes in the current version.
The author was partially supported by NSF grant number
DMS-1708320.


\begin{thebibliography}{8}

\bibitem{Appel}
  K. Appel and W. Haken,
  \emph{Every planar map is four colorable},
  volume 98 of \emph{Contemporary Mathematics}.
  American Mathematical Society, Providence, RI, 1989.

\bibitem{Blanchet}
  C.~Blanchet,
  N.~Habegger,
  G.~Masbaum, and
  P.~ Vogel,
  \emph{Topological quantum field theories derived from the Kauffman
    bracket},
  Topology,
  \textbf{34}(1995) 883--927.

\bibitem{Boozer}
  A.~D.~Boozer,
  Computer program for finding lower bounds on $J^\flat(K)$,
  \url{http://www.math.ucla.edu/~davidboozer}.

\bibitem{Brinkmann}
  G.~Brinkmann and W.~M.~Dress, 
  \emph{A Constructive Enumeration of Fullerenes},
  J. Algorithms,
  \textbf{23}(1997) 245--358.

\bibitem{Gonthier}
  G.~Gonthier,
  \emph{Formal Proof--The Four-Color Theorem},
  Notices of the American Mathematical Society,
  \textbf{55}(2008) 1382--139.

\bibitem{Khovanov}
  M.~Khovanov and L.-H.~Robert,
  \emph{Foam evaluation and Kronheimer--Mrowka theories},
  preprint,
  \textbf{arXiv:1808.09662}.

\bibitem{Kronheimer-1}
  P.~B.~Kronheimer and T.~S.~Mrowka,
  \emph{Tait colorings, and an instanton homology for webs and foams},
  preprint,
  \textbf{arXiv:1508.07205}.

\bibitem{Kronheimer-2}
  P.~B.~Kronheimer and T.~S.~Mrowka,
  \emph{Exact triangles for $SO(3)$ instanton homology of webs},
  preprint,
  \textbf{arXiv:1508.07207}.

\bibitem{Kronheimer-3}
  P.~B.~Kronheimer and T.~S.~Mrowka,
  \emph{A deformation of instanton homology for webs},
  preprint,
  \textbf{arXiv:1710.05002}.

\bibitem{Robert}
  L.-H.~Robert and E.~Wagner,
  \emph{A closed formula for the evaluation of $\mathfrak{sl}_n$-foams},
  preprint,
  \textbf{arXiv:1702.04140}.

\bibitem{Robertson}
  N.~Robertson, D.~Sanders, P.~Seymour, and R.~Thomas,
  \emph{The Four-Colour Theorem},
  J. Combin. Theory Ser. B,
  \textbf{70}(1997) 2--44.

\end{thebibliography}
\end{document}